\documentclass[12pt]{article}
\usepackage[hmargin=0.9in,vmargin=1in]{geometry}
\usepackage[T1]{fontenc}
\usepackage[utf8]{inputenc}
\usepackage{lmodern,microtype,mathtools,amsthm,amssymb,enumitem,xcolor,hyperref,bookmark}

\hypersetup{
    colorlinks=true,
    linkcolor={red!50!black},
    citecolor={blue!50!black},
    bookmarksopen=true,
    bookmarksnumbered,
    bookmarksopenlevel=2,
    bookmarksdepth=3
}

\usepackage[capitalise]{cleveref}
\usepackage{thm-restate}

\setlist[itemize]{leftmargin=2em,labelsep=.6em,topsep=\medskipamount,noitemsep}
\setlist[enumerate]{leftmargin=2em,labelsep=.5em,topsep=\medskipamount,noitemsep}

\newtheorem{theorem}{Theorem}[section]
\newtheorem{lemma}[theorem]{Lemma}
\newtheorem{corollary}[theorem]{Corollary}

\newtheorem{problem}[theorem]{Problem}
\newtheorem{claim}[theorem]{Claim}

\makeatletter
\newdimen\@savedtopsep
\renewenvironment{proof}[1][\proofname]{\par\@savedtopsep\topsep
\pushQED{\qed}\normalfont\topsep6\p@\@plus6\p@\trivlist
\item[\hskip\labelsep\itshape #1\@addpunct{.}]\topsep\@savedtopsep
\ignorespaces}{\popQED\endtrivlist\@endpefalse}
\makeatother

\newenvironment{claimproof}{\begin{proof}[Proof of Claim]}{\end{proof}}

\crefname{problem}{Problem}{Problems}
\crefname{claim}{Claim}{Claims}

\newcommand{\bbN}{\mathbb{N}}
\newcommand{\cG}{\mathcal{G}}
\newcommand{\cT}{\mathcal{T}}
\newcommand{\C}{\mathsf{C}}
\newcommand{\K}{\mathsf{K}}
\newcommand{\M}{\mathsf{M}}
\newcommand{\N}{\mathsf{N}}
\newcommand{\R}{\mathsf{R}}
\newcommand{\NP}{\textsf{NP}}
\newcommand{\param}{\mathsf{p}}

\DeclareMathOperator{\yolov}{tree\textnormal{-}\mu}
\DeclareMathOperator{\treealpha}{tree\textnormal{-}\alpha}

\DeclarePairedDelimiter{\set}{\{}{\}}
\DeclarePairedDelimiter{\size}{\lvert}{\rvert}

\let\leq\leqslant
\let\geq\geqslant
\let\setminus\smallsetminus

\title{Excluding a clique or a biclique in graphs of bounded induced matching treewidth}

\author{Tara Abrishami\thanks{University of Hamburg, Germany (\texttt{tara.abrishami@uni-hamburg.de}).
Supported by the National Science Foundation Award Number DMS-2303251 and the Alexander von Humboldt Foundation.}
\and Marcin Briański\thanks{Department of Theoretical Computer Science, Faculty of Mathematics and Computer Science, Jagiellonian University, Kraków, Poland (\texttt{marcin.brianski@doctoral.uj.edu.pl}, \texttt{bartosz.walczak@uj.edu.pl}).
Partially supported by Polish National Science Centre grant number 2019/34/E/ST6/00443.}
\and Jadwiga Czyżewska\thanks{University of Warsaw, Poland (\texttt{j.czyzewska@mimuw.edu.pl}).
Supported by Polish National Science Centre SONATA BIS-12 grant number 2022/46/E/ST6/00143.}
\and Rose McCarty\thanks{School of Mathematics and School of Computer Science, Georgia Institute of Technology, USA (\texttt{rmccarty3@gatech.edu}).
Supported by the National Science Foundation under Grant Number DMS-2202961.}
\and Martin Milanič\thanks{FAMNIT and IAM, University of Primorska, Slovenia (\texttt{martin.milanic@upr.si}).
Supported in part by the Slovenian Research and Innovation Agency (I0-0035, research program P1-0285 and research projects J1-3001, J1-3002, J1-3003, J1-4008, J1-4084, and N1-0370), and by the research program CogniCom (0013103) at the University of Primorska.}
\and Paweł Rzążewski\thanks{Warsaw University of Technology \& University of Warsaw, Poland (\texttt{pawel.rzazewski@pw.edu.pl}).
This work is a part of the project BOBR that has received funding from the European Research Council (ERC) under the European Union's Horizon 2020 research and innovation programme (grant agreement No.~948057).}
\and Bartosz Walczak\footnotemark[2]}

\date{}

\begin{document}

\maketitle

\begin{abstract}
For a tree decomposition $\cT$ of a graph $G$, let $\mu(\mathcal{T})$ denote the maximum size of an induced matching in $G$ with the property that some bag of $\mathcal{T}$ contains at least one endpoint of every edge of the matching.
The \emph{induced matching treewidth} of a graph $G$ is the minimum value of $\mu(\mathcal{T})$ over all tree decompositions $\mathcal{T}$ of~$G$.
Classes of graphs with bounded induced matching treewidth admit polynomial-time algorithms for a number of problems, including \textsc{Independent Set}, \textsc{$k$-Coloring}, \textsc{Odd Cycle Transversal}, and \textsc{Feedback Vertex Set}.
In this paper, we focus on combinatorial properties of such classes.

First, we show that graphs with bounded induced matching treewidth that exclude a fixed biclique as an induced subgraph have bounded \emph{tree-independence number}, which is another well-studied parameter defined in terms of tree decompositions.
This sufficient condition about excluding a biclique is also necessary, as bicliques have unbounded tree-independence number.
Second, we show that graphs with bounded induced matching treewidth that exclude a fixed clique have bounded chromatic number, that is, classes of graphs with bounded induced matching treewidth are $\chi$-bounded.
The two results confirm two conjectures due to Lima et~al.~[ESA 2024].
\end{abstract}

\section{Introduction}
The notions of tree decompositions and treewidth are among the most successful concepts in graph theory.
They were discovered by several groups of authors with various motivations~\cite{ACP87,BB72,Halin76,RS84}.
They proved to be essential in the study of graph minors~\cite{RS84,RS99,RS03} and in the design of efficient algorithms on graph classes~\cite[Chapters 7 and~11]{CFKLMPPS15}.

A tree decomposition describes the structure of a graph $G$ in terms of a tree whose nodes correspond to subsets of the vertices of $G$ called \emph{bags} (see \cref{sec:prelim} for a precise definition) in such a way that imposing local conditions on the bags leads to strong and interesting properties of the entire graph $G$.
For instance, bounding the size of every bag leads to the definition of treewidth.
As another example, chordal graphs are the graphs that admit a tree decomposition where all bags are cliques.
Such bags may be of arbitrary size, but their structure is heavily restricted.

A common generalization of both the above-mentioned restrictions is the concept of \emph{tree-independence number}, introduced independently by Yolov~\cite{Yolov18} and by Dallard, Milanič, and Štorgel~\cite{DMS24a}; we follow the notation of the latter paper.
The \emph{independence number} of a tree decomposition $\cT$, denoted by $\alpha(\cT)$, is the maximum size of an independent set contained in a single bag of $\cT$.
The \emph{tree-independence number} of a graph $G$, denoted by $\treealpha(G)$, is the minimum independence number of a tree decomposition of $G$.
Using this terminology, chordal graphs can be described as the graphs with tree-independence number equal to $1$.

Classes of graphs with bounded tree-independence number have received some attention, both from the structural~\cite{AACHSV24,CHLS24,DMS24b} and from the algorithmic~\cite{DFGKM24,DMS24a,LMMORS24,LMMORS-ESA24,Yolov18} points of view.
However, as pointed out already by Yolov~\cite{Yolov18}, tree-independence number is in some ways unsatisfactory.
For example, such a simple graph as $K_{t,t}$ (the complete bipartite graph with $t$ vertices on each side) has large tree-independence number.
Specifically, $\treealpha(K_{t,t})=t$, because every tree decomposition of $K_{t,t}$ has a bag that contains one of the two parts of the bipartition.

As a remedy, Yolov~\cite{Yolov18} defined a relaxed variant of $\treealpha$; we use the notation introduced later in~\cite{LMMORS24,LMMORS-ESA24}.
For a tree decomposition $\cT$ of a graph $G$, we say that an induced matching (i.e., a set of edges that are pairwise disjoint and non-adjacent in $G$) $M$ in $G$ \emph{touches} a bag $\beta$ of $\cT$ if every edge in $M$ has at least one end in $\beta$, and we let $\mu(\cT)$ denote the maximum size of an induced matching that touches some bag of $\cT$.
The \emph{induced matching treewidth} of $G$, denoted by $\yolov(G)$, is the minimum value of $\mu(\cT)$ over all tree decompositions $\cT$ of $G$.
While this definition may seem somewhat arbitrary, Yolov~\cite{Yolov18} showed that some classical \NP-hard problems, including \textsc{Independent Set} and $k$-\textsc{Coloring}, can be solved in polynomial time in classes of graphs with bounded $\yolov$.
Further algorithmic applications and structural results were later provided by Lima, Milanič, Muršič, Okrasa, Rzążewski, and Štorgel~\cite{LMMORS24,LMMORS-ESA24}.

There is a straightforward relationship between tree-independence number and induced matching treewidth: $\yolov(G)\leq\treealpha(G)$ for every graph $G$.
Indeed, if $\cT$ is a tree decomposition of $G$ and $M$ is an induced matching that touches a bag $\beta$ of $\cT$, then $\beta$ includes an independent set of size $\size{M}$ containing one end of every edge from $M$, witnessing that $\mu(\cT)\leq\alpha(\cT)$.
On the other hand, tree-independence number cannot be bounded from above by any function of induced matching treewidth.
For example, $\yolov(K_{t,t})=1$ (as $K_{t,t}$ contains no two-edge induced matching) while $\treealpha(K_{t,t})=t$, as we have observed above.

Lima et~al.~\cite[Conjecture~8.2]{LMMORS24} (see also~\cite[Conjecture~17]{LMMORS-ESA24}) 
conjectured that the presence of large complete bipartite graphs (also called \emph{bicliques}) is the \emph{only} reason why graphs of small $\yolov$ might have large $\treealpha$.
More formally, they suggested that graphs of bounded induced matching treewidth that exclude some fixed biclique as an induced subgraph actually have bounded tree-independence number.
As our first result, we confirm this conjecture.

\begin{restatable}{theorem}{thmbiclique}
\label{thm:biclique}
For any positive integers\/ $\mu$ and\/ $t$, there is an integer\/ $\K(\mu,t)$ such that the following holds.
Every\/ $K_{t,t}$-free graph\/ $G$ with\/ $\yolov(G)\leq\mu$ satisfies\/ $\treealpha(G)<\K(\mu,t)$.
\end{restatable}

This result, combined with the discussion above, immediately yields the following classification of tree-independence number and induced matching treewidth in hereditary classes of graphs, i.e., classes that are closed under vertex deletion.
(Here, by functionally equivalent, we mean that one parameter can be bounded by a function of the other, and vice versa.)

\begin{corollary}
\label{cor:funcEquiv}
The parameters\/ $\treealpha$ and\/ $\yolov$ are functionally equivalent in a hereditary graph class\/ $\cG$ if and only if\/ $\cG$ does not contain all complete bipartite graphs.
\end{corollary}

\Cref{thm:biclique} strengthens the following known result on graphs with $K_{t,t}$ excluded as a~subgraph rather than an induced subgraph: for any $\mu$ and $t$, there exists $\K'(\mu,t)$ such that every graph $G$ with $\yolov(G)\leq\mu$ and with no $K_{t,t}$ subgraph has treewidth less than $\K'(\mu,t)$ \cite[Theorem~7.5]{LMMORS24}.
This follows from \cref{thm:biclique}, because, by Ramsey's theorem, treewidth is bounded by a function of $\treealpha$ and the maximum size of a clique, and the latter is less than $2t$ in considered graphs.
Graphs with bounded $\yolov$ and with no $K_{t,t}$ subgraph are known to have bounded minimum degree.
This can be seen as a corollary to the above-mentioned result from~\cite{LMMORS24} (and the fact that minimum degree is bounded by treewidth), but it also follows from a well-known theorem of Kühn and Osthus~\cite{KO04} that graphs with no induced subdivisions of a fixed graph and with no $K_{t,t}$ subgraph have bounded minimum degree; this is because every proper subdivision of $K_{s,s}$ (with all edges subdivided at least once) has induced matching treewidth at least $s$.

In hereditary graph classes, bounded minimum degree implies bounded chromatic number.
If we exclude a large clique rather than a large $K_{t,t}$, we cannot hope for a property as strong as bounded minimum degree (as it does not hold for bicliques), but it turns out that we can still bound the chromatic number.
A graph class $\cG$ is \emph{$\chi$-bounded} if there is a function $f\colon\bbN\to\bbN$ such that every graph $G\in\cG$ satisfies $\chi(G)\leq f(\omega(G))$.
Here, $\chi(G)$ is the chromatic number of $G$, and $\omega(G)$ is the maximum size of a clique in $G$.
The study of $\chi$-boundedness is a very fruitful topic in graph theory; see~\cite{SS20} for a survey.
Since chromatic number is bounded by treewidth (plus one), which is in turn bounded by a function of $\treealpha$ and $\omega$, classes of graphs with bounded tree-independence number are $\chi$-bounded.

Another conjecture of Lima et~al.~\cite[Conjecture~8.3]{LMMORS24} (see also~\cite[Conjecture~18]{LMMORS-ESA24}) asserts that classes of graphs with bounded $\yolov$ are also $\chi$-bounded.
Our second result confirms this conjecture.

\begin{restatable}{theorem}{thmchibdd}
\label{thm:chibdd}
Every class of graphs with bounded\/ $\yolov$ is\/ $\chi$-bounded.
\end{restatable}

The rest of the paper is organized as follows.
In \cref{sec:prelim}, we formally define tree decompositions and the parameters considered, and we introduce the notation used further on.
We prove \cref{thm:biclique} in \cref{sec:biclique} and \cref{thm:chibdd} in \cref{sec:chibdd}.
Finally, in \cref{sec:open}, we suggest open problems for future work.

\section{Preliminaries}\label{sec:prelim}
We let $\bbN$ denote the set of positive integers.
For $n\in\bbN$, we let $[n]$ denote the set $\set{1,\ldots,n}$.

\paragraph{Graph theory}
Let $G$ be a graph.
We let $V(G)$ denote the set of vertices and $E(G)$ denote the set of edges of $G$.
For a vertex $v$ of $G$, we let $N_G(v)$ denote the set of neighbors of $v$, and we let $N_G[v]$ denote $N_G(v)\cup\set{v}$.
Similarly, for a set $X\subseteq V(G)$, we define $N_G[X]=\bigcup_{v\in X}N_G[v]$ and $N_G(X)=N_G[X]\setminus X$.
If $G$ is clear from the context, we drop the subscript in the notation above and simply write $N(\cdot)$ and $N[\cdot]$.

A \emph{clique} in $G$ is a set of vertices of $G$ that are pairwise adjacent, and an \emph{independent set} in $G$ is a set of vertices of $G$ that are pairwise nonadjacent.
We let $\omega(G)$ denote the maximum size of a clique and $\alpha(G)$ denote the maximum size of an independent set in $G$.
A \emph{proper coloring} of $G$ is an assignment of colors to the vertices of $G$ such that no two adjacent vertices receive the same color.
We let $\chi(G)$ denote the \emph{chromatic number} of $G$, i.e., the minimum number of colors in a proper coloring of $G$.
A \emph{matching} in $G$ is a set of pairwise disjoint edges of $G$.
Kőnig's theorem asserts that in every bipartite graph $G$, the maximum size of a matching is equal to the minimum size of a vertex cover, where the latter is equal to $\size{V(G)}-\alpha(G)$, as the vertex covers are exactly the complements of the independent sets.

\paragraph{Induced subgraphs}
For a set $X \subseteq V(G)$, we let $G[X]$ denote the subgraph of $G$ \emph{induced by} $X$, i.e., the graph with vertex set $X$ and edge set $\set{uv\in E(G)\mid u,v\in X}$.
We say that $H$ is an \emph{induced subgraph} of $G$ if a copy of $H$ is induced in $G$ by some subset of $V(G)$.
An \emph{induced matching} in $G$ is a matching that occurs in $G$ as the edge set of an induced subgraph, which means that the vertices of distinct edges of the matching are not only distinct but also nonadjacent.
For a positive integer $t$, we let $K_{t,t}$ denote the complete bipartite graph with $t$ vertices on each side of the bipartition, and we say that $G$ is \emph{$K_{t,t}$-free} if $K_{t,t}$ is not an induced subgraph of $G$.

If the host graph $G$ is clear from the context, we sometimes identify vertex sets with the subgraphs induced by them.
For example, we say that a subset $X$ of $V(G)$ is connected to indicate that $G[X]$ is connected, and we write $\alpha(X)$, $\chi(X)$, etc.\ for $\alpha(G[X])$, $\chi(G[X])$, etc., respectively.

\paragraph{Tree decompositions}
A \emph{tree decomposition} of a graph $G$ is a pair $\cT=(T,\beta)$ where $T$ is a~tree and $\beta$ is a function that maps the vertices of $T$ to subsets of $V(G)$ so that the following three conditions are satisfied.
\begin{enumerate}
\item For every vertex $v$ of $G$, there is $x\in V(T)$ such that $v\in\beta(x)$.
\item For every edge $uv$ of $G$, there is $x\in V(T)$ such that $u,v\in\beta(x)$.
\item For every vertex $v$ of $G$, the set $\set{x\in V(T)\mid v\in\beta(x)}$ induces a connected subgraph of $T$ (i.e., a subtree).
We let $T_v$ denote this subtree in the context of a fixed tree decomposition.
\end{enumerate}
To avoid confusion between vertices of $G$ and of $T$, we refer to the latter as \emph{nodes}.
For each node $x$, its corresponding set $\beta(x)$ is called a \emph{bag}.

Let $G$ be a graph, and let $\cT=(T,\beta)$ be a tree decomposition of $G$.
The \emph{independence number} of $\cT$, denoted by $\alpha(\cT)$, is defined as $\max_{x\in V(T)}\alpha(\beta(x))$.
An induced matching $M$ in $G$ is said to \emph{touch} a set $X\subseteq V(G)$ if every edge in $M$ has an endpoint in $X$.
The \emph{induced matching treewidth} of $\cT$, denoted by $\mu(\cT)$, is defined as the maximum over $x\in V(T)$ of the maximum size of an induced matching that touches $\beta(x)$.
The minimum values of $\alpha(\cT)$ and $\mu(\cT)$ over all tree decompositions $\cT$ of $G$ are called the \emph{tree-independence number} and the \emph{induced matching treewidth} of $G$ and are denoted by $\treealpha(G)$ and $\yolov(G)$, respectively.
It is straightforward to verify that $\mu(\cT)\leq\alpha(\cT)$ for every tree decomposition $\cT$ of $G$, and thus $\yolov(G)\leq\treealpha(G)$.

\paragraph{Ramsey numbers}
For positive integers $t_1,\ldots,t_k$, we let $\R(t_1,\ldots,t_k)$ denote the \emph{Ramsey number}, i.e., the minimum integer $n$ such that for every $k$-coloring of the edges of the complete graph on $n$ vertices, we can find a $t_i$-vertex complete subgraph all of whose edges are colored $i$, for some $i\in[k]$.
The existence of such a number is a consequence of Ramsey's theorem.

\section{Proof of \texorpdfstring{\cref{thm:biclique}}{Theorem \ref{thm:biclique}}}\label{sec:biclique}
We start with two auxiliary Ramsey-type results.
The first one is well known (see, e.g., \cite{DDL13}), but we include a short proof for completeness.

\begin{lemma}[folklore]\label{lem:extract-induced-matching}
For any positive integers\/ $s$ and\/ $t$, there is an integer\/ $\M(s,t)$ for which the following holds.
Every bipartite graph that contains a matching of size\/ $\M(s,t)$ contains either an induced\/ $K_{t,t}$ or an induced matching of size\/ $s+1$.
\end{lemma}

\begin{proof}
Let $n=\M(s,t)=\R(2t,2t,s+1)$.
Let $G$ be a bipartite graph with vertex sets $A$ and $B$ on the two sides.
Let $M=\set{a_1b_1,\ldots,a_nb_n}$ be a matching of size $n$ in $G$, where $a_1,\ldots,a_n\in A$ and $b_1,\ldots,b_n\in B$.
Consider an auxiliary complete graph $K$ with vertex set $[n]$.
Color the edges of this graph with three colors as follows: for any $i,j\in[n]$ with $i<j$, let the color of the edge $ij$ be
\begin{itemize}
\item $1$ if $a_ib_j\in E(G)$,
\item $2$ if $a_ib_j\notin E(G)$ and $a_jb_i\in E(G)$,
\item $3$ if $a_ib_j\notin E(G)$ and $a_jb_i\notin E(G)$.
\end{itemize}
Since $n=\R(2t,2t,s+1)$, at least one of the following cases occurs.
\begin{enumerate}
\item There is a clique of size $2t$ with all edges of color $1$.
Such a clique on increasing indices $i_1,\ldots,i_t,j_1,\ldots,j_t$ implies an induced $K_{t,t}$ in $G$ on the vertices $a_{i_1},\ldots,a_{i_t}$ and $b_{j_1},\ldots,b_{j_t}$.
\item There is a clique of size $2t$ with all edges of color $2$.
Such a clique on increasing indices $j_1,\ldots,j_t,i_1,\ldots,i_t$ implies an induced $K_{t,t}$ in $G$ on the vertices $a_{i_1},\ldots,a_{i_t}$ and $b_{j_1},\ldots,b_{j_t}$.
\item There is a clique of size $s+1$ with all edges of color $3$.
Such a clique on increasing indices $i_1,\ldots,i_{s+1}$ implies an induced matching $\set{a_{i_1}b_{i_1},\ldots,a_{i_{s+1}}b_{i_{s+1}}}$ of size $s+1$ in $G$.\qedhere
\end{enumerate}
\end{proof}

In the next lemma, we show that from a number of \emph{very large} independent sets in a $K_{t,t}$-free graph, we can extract \emph{large} subsets such that their union is independent.

\begin{lemma}\label{lem:extract-independent-set}
For any positive integers\/ $s$, $t$, and\/ $m$, there is an integer\/ $\N(s,t,m)$ for which the following holds.
Let\/ $G$ be a\/ $K_{t,t}$-free graph, and let\/ $I_1,\ldots,I_m$ be independent sets in\/ $G$ each of size at least\/ $\N(s,t,m)$.
Then there is an independent set\/ $I$ in\/ $G$ such that\/ $\size{I\cap I_i}\geq s$ for all\/ $i\in[m]$.
\end{lemma}

\begin{proof}
Let $n=\N(s,t,m)=\R(\overbrace{2t,\ldots,2t}^{\binom{m}{2}},ms)$.
For each $i\in[m]$, let $v^i_1,\ldots,v^i_n$ be arbitrary $n$ vertices from $I_i$.
Consider an auxiliary complete graph $K$ with vertex set $[n]$.
Color the edges of this graph with $\binom{m}{2}+1$ colors as follows: for any $k,\ell\in[n]$ with $k<\ell$, if there are indices $i,j\in[m]$ with $i<j$ such that $v^i_kv^j_\ell\in E(G)$, then color the edge $k\ell$ with such a pair $(i,j)$, otherwise color the edge $k\ell$ with a special color $c$.
It follows that $K$ contains a clique of size $2t$ with all edges of color $(i,j)$ for some $i,j\in[m]$ with $i<j$ or a clique of size $ms$ with all edges of the special color $c$.
In the former case, a clique in $K$ on increasing indices $k_1,\ldots,k_t,\ell_1,\ldots,\ell_t$ with all edges of color $(i,j)$ implies an induced $K_{t,t}$ in $G$ on the vertices $v^i_{k_1},\ldots,v^i_{k_t}$ and $v^j_{\ell_1},\ldots,v^j_{\ell_t}$, which contradicts the assumption that $G$ is $K_{t,t}$-free.
Thus $K$ contains a clique of size $ms$ on some increasing indices $k^1_1,\ldots,k^1_s,\ldots,k^m_1,\ldots,k^m_s$ with all edges of color $c$.
For each $i\in[m]$, let $I'_i=\set{v^i_{\smash[t]{k^i_1}},\ldots,v^i_{\smash[t]{k^i_s}}}$.
It follows that $v^i_kv^j_\ell\notin E(G)$ for all $i,j\in[m]$ and all $v^i_k\in I'_i$ and $v^j_\ell\in I'_j$.
We let $I=I'_1\cup\cdots\cup I'_m$ and conclude that $I$ is an independent set in $G$ with $\size{I\cap I_i}\geq\size{I'_i}=s$ for all $i\in[m]$.
\end{proof}

Now, we proceed to the main theorem of this section.

\thmbiclique*

\begin{proof}
Let $\M(\cdot,\cdot)$ and $\N(\cdot,\cdot,\cdot)$ be as claimed in \cref{lem:extract-induced-matching,lem:extract-independent-set}, respectively.
Let
\begin{align*}
\C(\mu,t)&=\N(\M(\mu,t),\:t,\:\M(\mu,t))\text{,} \\
\K(\mu,t)&=2\cdot\M(\mu,t)+\mu\cdot\C(\mu,t)\text{.}
\end{align*}
Let $G$ be a $K_{t,t}$-free graph, and let $\cT=(T,\beta)$ be a tree decomposition of $G$ with $\mu(\cT)\leq\mu$.
We aim to show that $\treealpha(G)<\K(\mu,t)$.
Let $S$ be a maximum independent set in $G$.

\begin{claim}
\label{clm:max-ind-set-helper}
For every node\/ $x$ of\/ $T$, it holds that\/ $\alpha(\beta(x)\setminus S)<\M(\mu,t)$.
\end{claim}

\begin{claimproof}
Fix a node $x$ of $T$, and let $I$ be a maximum independent set in $\beta(x)\setminus S$.
Since the graph $G[S\cup I]$ is bipartite with parts $S$ and $I$, Kőnig's theorem guarantees a matching of size $\size{S\cup I}-\alpha(S\cup I)$ in $G[S\cup I]$.
Since $S$ is a maximum independent set in $G$ and therefore in $S\cup I$, the aforesaid matching in $G[S\cup I]$ has size $\size{S\cup I}-\size{S}=\size{I}$.
Suppose, towards a contradiction, that $\size{I}\geq\M(\mu,t)$.
Since $G$ is $K_{t,t}$-free, \cref{lem:extract-induced-matching} guarantees an induced matching of size at least $\mu+1$ in $G[S\cup I]$.
This matching touches $\beta(x)$, contrary to the assumption that $\mu(\cT)\leq\mu$.
We conclude that $\alpha(\beta(x)\setminus S)=\size{I}<\M(\mu,t)$.
\end{claimproof}

Let a vertex $v$ of $G$ be called \emph{light} if $\alpha(N(v))<\C(\mu,t)$ and \emph{heavy} otherwise.
Let $S_\ell$ and $S_h$ be the subsets of $S$ comprising the light vertices and the heavy vertices, respectively, so that $S=S_\ell\cup S_h$.
In the next two claims, we analyze the interaction of the vertices from $S_\ell$ and $S_h$ with a single bag of $\cT$.

\begin{claim}\label{clm:light}
For every node\/ $x$ of\/ $T$, it holds that\/ $\alpha(N(\beta(x)\cap S_\ell))<\mu\cdot\C(\mu,t)$.
\end{claim}

\begin{claimproof}
Fix a node $x$ of $T$, and let $S^x_\ell=\beta(x)\cap S_\ell$.
Let $I$ be a maximum independent set in $N(S^x_\ell)$.
We aim to show that $\size{I}<\mu\cdot\C(\mu,t)$.
Let $U$ be an inclusion-minimal subset of $S^x_\ell$ such that $I\subseteq N(U)$.
By the minimality of $U$, for every $s\in U$, there is a vertex $v_s\in I$ such that $N(v_s)\cap U=\set{s}$.
Since both $U$ and $I$ are independent sets and $U\subseteq\beta(x)$, the set $\set{sv_s\mid s\in U}$ is an induced matching of size $\size{U}$ in $G$ that touches $\beta(x)$.
Consequently, $\size{U}\leq\mu(\cT)\leq\mu$.
Since $U\subseteq S^x_\ell\subseteq S_\ell$, all vertices in $U$ are light, whence it follows that $\size{I\cap N(s)}<\C(\mu,t)$ for every $s\in U$.
We conclude that
\[\alpha(N(\beta(x)\cap S_\ell))=\size{I}\leq\sum_{s\in U}\size{I\cap N(s)}<\size{U}\cdot\C(\mu,t)\leq\mu\cdot\C(\mu,t)\text{.}\qedhere\]
\end{claimproof}

\begin{claim}\label{clm:heavy}
For every node\/ $x$ of\/ $T$, it holds that\/ $\size{\beta(x)\cap S_h}<\M(\mu,t)$.
\end{claim}

\begin{claimproof}
Fix a node $x$ of $T$, and let $S^x_h=\beta(x)\cap S_h$.
Let $m=\M(\mu,t)$.
We aim to show that $\size{S^x_h}<m$.
Suppose, to the contrary, that $\size{S^x_h}\geq m$, and let $s_1,\ldots,s_m$ be arbitrary $m$ vertices from $S^x_h$.
For each $i\in[m]$, let $I_i$ be a maximum independent set in $N(s_i)$;
since the vertex $s_i$ is heavy, we have $\size{I_i}\geq\C(\mu,t)=\N(m,t,m)$.
\Cref{lem:extract-independent-set} applied to $I_1,\ldots,I_m$ provides an independent set $I$ in $G$ such that $\size{I\cap I_i}\geq m$ for all $i\in[m]$.
Consider the bipartite graph $H=G[\set{s_1,\ldots,s_m}\cup I]$.
Since $\size{N_H(s_i)}=\size{I\cap I_i}\geq m$ for all $i\in[m]$, we can greedily find a~matching of size $m$ in $H$.
Since $H$ is $K_{t,t}$-free, \cref{lem:extract-induced-matching} provides an induced matching of size $\mu+1$ in $H$.
However, since $s_1,\ldots,s_m\in\beta(x)$, such a matching touches the bag $\beta(x)$, contrary to the assumption that $\mu(\cT)\leq\mu$.
We conclude that $\size{S^x_h}<m$.
\end{claimproof}

Recall that for a vertex $v$ of $G$, the subgraph of $T$ induced by the nodes that contain $v$ in their bags is denoted by $T_v$; since $\cT$ is a tree decomposition, $T_v$ is a nonempty tree.
We now construct a tree decomposition $\cT'=(T',\beta')$ of $G$ as follows.
\begin{itemize}
\item The tree $T'$ is obtained from $T$ by adding, for every $s\in S_\ell$, a new leaf node $y_s$ adjacent to some node $x_s$ of $T_s$.
\item For every node $x$ of $T$, we set $\beta'(x)=(\beta(x) \setminus S_\ell)\cup N(\beta(x)\cap S_\ell)$.
\item For every vertex $s\in S_\ell$, we set $\beta'(y_s)=N[s]$.
\end{itemize}
In the next two claims, we show that $\cT'$ is a tree decomposition of $G$ and $\alpha(\cT')<\K(\mu,t)$.

\begin{claim}\label{clm:treedec}
$\cT'$ is a tree decomposition of\/ $G$.
\end{claim}

\begin{claimproof}
We show that $\cT'$ satisfies the three conditions on a tree decomposition listed in \cref{sec:prelim}, that is, every vertex of $G$ is in a bag, every edge of $G$ is in a bag, and for every vertex $v$ of $G$, the subgraph $T'_v$ of $T'$ induced by the nodes $x$ with $v\in\beta'(x)$ is connected.

Consider a vertex $v$ of $G$.
If $v\in S_\ell$, then $v\in\beta'(y_v)$.
If $v\notin S_\ell$, then, since $\cT$ is a tree decomposition of $G$, there is a node $x$ of $T$ such that $v\in\beta(x)$ and thus $v\in\beta(x)\setminus S_\ell\subseteq\beta'(x)$.

Next, consider an edge $uv$ of $G$.
If $\set{u,v}\cap S_\ell=\emptyset$, then there is a node $x$ of $T$ such that $u,v\in\beta(x)$, and consequently $u,v\in\beta(x)\setminus S_\ell\subseteq\beta'(x)$.
If $\set{u,v}\cap S_\ell\neq\emptyset$, then there is unique $s\in\set{u,v}\cap S_\ell$, and we have $\set{u,v}\subseteq N[s]=\beta'(y_s)$.

Finally, fix a vertex $v$ of $G$, and consider the subgraph $T'_v$ of $T'$.
If $v\in S_\ell$, then $V(T'_v)=\set{y_s}$, and so $T'_v$ is connected.
If $v\in V(G)\setminus N[S_\ell]$, then $T'_v=T_v$, and so $T'_v$ is connected, as $\cT$ is a~tree decomposition of $G$.
Now, suppose $v\in N(S_\ell)$.
In this case,
\[V(T'_v)=V(T_v)\cup\bigcup_{s\in S_\ell\cap N(v)}(V(T_s)\cup\set{y_s})\text{.}\]
By the construction of $\cT'$, for each $s\in S_\ell$, the subgraph of $T'$ induced by $V(T_s)\cup\set{y_s}$ is connected.
Furthermore, since $\cT$ is a tree decomposition of $G$, for each $s\in S_\ell\cap N(v)$, there is a node $x$ of $T$ such that $s,v\in\beta(x)$.
Therefore, $x\in V(T_v)\cap V(T_s)$, and in particular the set $V(T_v) \cup V(T_s)\cup\set{y_s}$ induces a connected subgraph of $T'$.
Consequently, $T'_v$ is connected.

Summing up, $\cT'$ is indeed a tree decomposition of $G$.
\end{claimproof}

\begin{claim}\label{clm:treealpha}
$\alpha(\cT')<\K(\mu,t)$.
\end{claim}

\begin{claimproof}
Consider a node $x$ of $T'$.
If $x=y_s$ for some $s\in S_\ell$, then, since $s$ is light, we have $\alpha(N(s))<\C(\mu,t)$ and therefore
\[\alpha(\beta'(x))=\alpha(N[s])<\C(\mu,t)\leq\K(\mu,t)\text{.}\]
Now, suppose that $x$ is a node of $T$.
Since $S=S_h\cup S_\ell$, it follows that
\[\beta'(x)=(\beta(x)\setminus S_\ell)\cup N(\beta(x)\cap S_\ell)=(\beta(x)\setminus S)\cup(\beta(x)\cap S_h)\cup N(\beta(x)\cap S_\ell)\text{.}\]
We bound the maximum size of an independent set in each of the three terms above separately.
By \cref{clm:max-ind-set-helper}, we have $\alpha(\beta(x)\setminus S)<\M(\mu,t)$.
By \cref{clm:heavy}, we have $\size{\beta(x)\cap S_h}<\M(\mu,t)$.
Finally, by \cref{clm:light}, we have $\alpha(N(\beta(x)\cap S_\ell))<\mu\cdot\C(\mu,t)$.
Summing up, we conclude that
\[\alpha(\beta'(x))<2\cdot\M(\mu,t)+\mu\cdot\C(\mu,t)=\K(\mu,t)\text{.}\qedhere\]
\end{claimproof}

Now, the theorem follows directly from \cref{clm:treedec,clm:treealpha}.
\end{proof}

\section{Proof of \texorpdfstring{\cref{thm:chibdd}}{Theorem \ref{thm:chibdd}}}\label{sec:chibdd}
A \emph{layering} of a graph $G$ is a finite sequence of sets $L_0,\ldots,L_t$ with the following properties:
\begin{itemize}
\item they form a partition of the vertex set of $G$, that is, $\bigcup_{i=0}^tL_i=V(G)$ and $L_i\cap L_j=\emptyset$ for $i\neq j$;
\item for every edge $uv$ of $G$, if $u\in L_i$ and $v\in L_j$, then $\size{i-j}\leq 1$.
\end{itemize}
The following well-known lemma asserts that in order to bound the chromatic number of $G$, it is sufficient to focus on the graphs induced by individual layers of a layering of $G$.
Although its proof is quite standard (see, e.g.,~\cite{Gyarfas85}), we include it for completeness.

\begin{lemma}[folklore]\label{lem:layering}
Let\/ $G$ be a graph and\/ $L_0,\ldots,L_t$ be a layering of\/ $G$.
Then
\[\chi(G)\leq 2\max\set{\chi(L_0),\ldots,\chi(L_t)}\text{.}\]
\end{lemma}

\begin{proof}
Let $k=\max\set{\chi(L_0),\ldots,\chi(L_t)}$.
For each $i\in\set{0,\ldots,t}$, let $\varphi_i\colon L_i\to[k]$ be a proper coloring of $G[L_i]$ with $k$ colors.
Define a coloring $\varphi$ of $G$ by setting $\varphi(v)=(\varphi_i(v),\:i\bmod 2)$, where $i$ is the unique index in $\set{0,\ldots,t}$ such that $v\in L_i$.
We claim that $\varphi$ is a proper coloring of $G$.
Indeed, consider an edge $uv$ of $G$, and let $i$ and $j$ be the unique indices with $u\in L_i$ and $v\in L_j$.
If $i=j$, then $\varphi_i(u)\neq\varphi_i(v)$, so $\varphi(u)\neq\varphi(v)$.
Otherwise, we have $\size{i-j}=1$, so $i\not\equiv j\pmod 2$, which implies $\varphi(u)\neq\varphi(v)$.
\end{proof}

The basic strategy in our proof of \cref{thm:chibdd} is to consider a variant of ``breadth-first search layering'' obtained by starting from two adjacent vertices rather than a single vertex, and to focus on coloring a single layer and understanding how it relates to the (fixed) tree decomposition of small $\yolov$.
Such a single layer is colored by splitting it into two sets -- one corresponding to part of the tree decomposition ``shared'' with prior layers of the layering and the other corresponding to what is decomposed ``on the fringes'' of the decomposition.

\thmchibdd*

\begin{proof}
We prove that there is a function $f\colon(\bbN\cup\set{0})\times\bbN\to\bbN$ such that every graph $G$ with $\yolov(G)\leq\mu$ and $\omega(G)\leq\omega$ satisfies $\chi(G)\leq f(\mu,\omega)$.
We proceed by induction on $\mu+\omega$.
The base cases are clear: every graph with $\yolov(G)=0$ or $\omega(G)=1$ is edgeless and thus can be properly colored with one color.
Now, fix two integers $\mu\geq 1$ and $\omega\geq 2$, and assume that the values $f(\mu',\omega')$ have the required property whenever $\mu'+\omega'<\mu+\omega$.
Let $G$ be a graph with $\yolov(G)\leq\mu$ and $\omega(G)\leq\omega$.
Since our task is to construct a proper coloring of $G$, we can assume that $G$ is connected (otherwise we can color each connected component separately) and has at least two vertices.
Observe that for every vertex $v$ of $G$, since $\omega(N(v))\leq\omega-1$, the induction hypothesis implies that $\chi(N(v))\leq f(\mu,\omega-1)$.

Fix a tree decomposition $\cT=(T,\beta)$ of $G$ with $\mu(\cT)\leq\mu$.
Let $u_0v_0$ be an arbitrary edge of $G$.
Let $L_0=\set{u_0,v_0}$, and for $i\geq 1$, by induction, let $L_i=N(L_0\cup\cdots\cup L_{i-1})$.
Let $t$ be maximum such that $L_t\neq\emptyset$.
It follows that $L_0,\ldots,L_t$ is a layering of $G$, and the set $L_0\cup\cdots\cup L_i$ is connected for every $i\in\set{0,\ldots,t}$.
In view of \cref{lem:layering}, we aim to bound $\chi(L_i)$ for all $i\in\set{0,\ldots,t}$.
Since $L_0=\set{u_0,v_0}$, we have $\chi(L_0)=2$.
Since $L_1\subseteq N(u_0)\cup N(v_0)$, we have $\chi(L_1)\leq\chi(N(u_0))+\chi(N(v_0))\leq 2f(\mu,\omega-1)$.
It remains to bound $\chi(L_i)$ when $i\in\set{2,\ldots,t}$.

Let $L'=L_0\cup\cdots\cup L_{i-2}$, and let $T'$ be the subtree of $T$ obtained as the union of all $T_v$ with $v\in L'$.
In other words, $T'$ is the subgraph of $T$ induced by all nodes whose bags intersect the set $L'$.
Since the set $L'$ is connected, $T'$ is indeed a subtree.
Let $A$ be the subset of $L_i$ comprising the vertices $v\in L_i$ such that $T_v$ and $T'$ share a node, and let $B=L_i\setminus A$.

\begin{claim}
$\chi(A)\leq f(\mu-1,\omega)$.
\end{claim}

\begin{claimproof}
It suffices to argue that $\yolov(A)\leq\mu-1$, because then the induction hypothesis entails $\chi(A)\leq f(\mu-1,\omega)$.
To this end, consider the tree decomposition $\cT'=(T',\beta')$ where $\beta'(x)=\beta(x)\cap A$ for all nodes $x$ of $T'$.
This is a tree decomposition of $G[A]$ thanks to the Helly property of subtrees of a fixed tree: if $uv$ is an edge of $G[A]$, then the trees $T_u$, $T_v$, and $T'$ pairwise intersect, and hence they have a common node.
Suppose, towards a contradiction, that there is a node $x$ of $T'$ such that the bag $\beta'(x)$ is touched by some induced matching $M$ of size $\mu$.
The node $x$ belongs to $T'$, so it also belongs to $T_v$ for some $v\in L'$, which is equivalent to saying that $v\in\beta(x)$.
Since the set $L'$ is connected and has size at least $2$, the vertex $v$ has a neighbor $u$ in $L'$.
Since there is no edge between $L'$ and $L_i$, the set $M\cup\set{uv}$ is an induced matching of size $\mu+1$ in $G$ that touches $\beta(x)$, contrary to the assumption that $\mu(\cT)\leq\mu$.
We conclude that $\yolov(A)\leq\mu(\cT')\leq\mu-1$, as desired.
\end{claimproof}

We also need to bound the chromatic number of $B$.
To this end, consider a connected component $C$ of $G[B]$ with maximum chromatic number, so that $\chi(B)=\chi(C)$.
Let $D$ be an inclusion-minimal subset of $L_{i-1}$ such that $C\subseteq N(D)$; it exists, as $C\subseteq N(L_{i-1})$.

\begin{claim}\label{clm:dominating}
$\size{D}<\R(\omega+1,\omega+1,\mu+1)$.
\end{claim}

\begin{claimproof}
Suppose, to the contrary, that $\size{D}\geq\R(\omega+1,\omega+1,\mu+1)$.
By the minimality of $D$, for every $d\in D$, there is a vertex $v_d\in C$ such that $N(v_d)\cap D=\set{d}$.

We claim that there is a node $x$ of $T$ such that $D\subseteq\beta(x)$.
To see this, let $T_C$ be the subgraph of $T$ induced by the union of the node sets of all trees $T_v$ with $v\in C$.
Since $C$ is connected, $T_C$ is a tree.
Moreover, the trees $T_C$ and $T'$ are node-disjoint.
Let $x$ be the node of $T'$ closest to $T_C$.
For every vertex $d\in D$, since $d$ has a neighbor in $C$ (namely, $v_d$) and a neighbor in $L_{i-2}$, the tree $T_d$ intersects both the tree $T_C$ and the tree $T'$; in particular, it contains the node $x$, which is equivalent to saying that $d\in\beta(x)$.

Now, consider an auxiliary complete graph $K$ with vertex set $D$.
Color the edges of this graph with three colors assigning the following color to the edge $dd'$ for all distinct $d,d'\in D$:
\begin{itemize}
\item $1$ if $dd'\in E(G)$,
\item $2$ if $dd'\notin E(G)$ and $v_dv_{d'}\in E(G)$,
\item $3$ if $dd'\notin E(G)$ and $v_dv_{d'}\notin E(G)$.
\end{itemize}
Since $\size{D}\geq\R(\omega+1,\omega+1,\mu+1)$, at least one of the following cases occurs:
\begin{enumerate}
\item There is a clique of size $\omega+1$ in $K$ with all edges of color $1$.
Such a clique corresponds to a~clique in $D$, contrary to the assumption that $\omega(G)\leq\omega$.
\item There is a clique of size $\omega+1$ in $K$ with all edges of color $2$.
Such a clique corresponds to a~clique in $\set{v_d\mid d\in D}$, again contradicting the assumption that $\omega(G)\leq\omega$.
\item There is a clique $D'$ of size $\mu+1$ in $K$ with all edges of color $3$.
Let $M=\set{dv_d\mid d\in D'}$.
Since $N(v_d)\cap D=\set{d}$ for every $d\in D$, it follows that $M$ is an induced matching of size $\mu+1$ in $G$.
Since $D\subseteq\beta(x)$, the matching $M$ touches the bag $\beta(x)$, which contradicts the assumption that $\mu(\cT)\leq\mu$.\qedhere
\end{enumerate}
\end{claimproof}

We have $C\subseteq N(D)\subseteq\bigcup_{d\in D}N(d)$, which implies that
\[\chi(B)=\chi(C)\leq\sum_{d\in D}\chi(N(d))\leq\size{D}\cdot f(\mu,\omega-1)\leq(\R(\omega+1,\omega+1,\mu+1)-1)\cdot f(\mu,\omega-1)\text{,}\]
where in the last inequality, we invoke \cref{clm:dominating}.

To sum up, we have shown that
\begin{gather*}
\chi(L_0)=2\text{,} \qquad\qquad \chi(L_1)\leq 2f(\mu,\omega-1)\text{,} \\
\chi(L_i)\leq\chi(A)+\chi(B)\leq f(\mu-1,\omega)+(\R(\omega+1,\omega+1,\mu+1)-1)\cdot f(\mu,\omega-1) \quad \text{for $i\geq 2$.}
\end{gather*}
Since $\mu\geq 1$ and $\omega\geq 2$, we have $\R(\omega+1,\omega+1,\mu+1)\geq 6$.
We can apply \cref{lem:layering} and conclude that $\chi(G)\leq f(\mu,\omega)$, setting $f(\mu,\omega)$ to be twice the right side of the last inequality above.
\end{proof}

\section{Open problems}\label{sec:open}
A possible way of generalizing our results is to consider some parameter $\param$ that is more general than induced matching treewidth, i.e., for every graph it holds that $\param(G)$ is bounded by some function of $\yolov(G)$ but not the other way around.

A possible choice of such a parameter is \emph{sim-width}, which is also defined in terms of induced matchings, but this time in \emph{branch decompositions} (see~\cite{KKST17}).
It is known that the sim-width is always bounded from above by $\yolov$ (see~\cite[Theorem~24]{BKR23}).
Relationships between various parameters, including sim-width, in restricted graph classes were recently studied by Brettell, Munaro, Paulusma, and Yang (see~\cite{BMPY23}).
However, they were not able to resolve the following problem, which we repeat here.

\begin{problem}[{\cite[Open Problem~1]{BMPY23}}]\label{prob:simw-biclique}
Do graph classes of bounded sim-width that exclude some fixed biclique as an induced subgraph have bounded tree-independence number?
\end{problem}

By \cref{thm:biclique}, \cref{prob:simw-biclique} is equivalent to the following.

\begin{problem}\label{prob:simw-biclique-2}
Do graph classes of bounded sim-width that exclude some fixed biclique as an induced subgraph have bounded induced matching treewidth?
\end{problem}

Actually, we believe that a stronger variant of \cref{prob:simw-biclique-2} might also have a positive answer.
For a positive integer $t$, let a \emph{$t$-obstruction} be any graph whose vertex set consists of four independent sets $A$, $B$, $C$, and $D$, each of size $t$, such that each of sets $A\cup B$ and $C\cup D$ induces a matching of $t$ edges and the set $B\cup C$ induces a copy of $K_{t,t}$; the other edges are arbitrary (see \Cref{fig:t-obstruction}).
We remark that any $t$-obstruction has induced matching treewidth at least $t$.
Hence, the presence of a $t$-obstruction, for large $t$, forces large induced matching treewidth, just like the presence of a large induced biclique forces large tree-independence number.

\begin{figure}[h]
    \centering
    \includegraphics[width=0.4\textwidth]{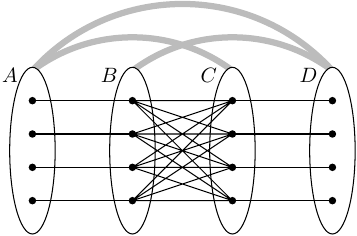}
    \caption{A $t$-obstruction for $t=4$.
    Edges between pairs of independent sets depicted by the three shaded arcs are arbitrary.}
    \label{fig:t-obstruction}
\end{figure}

\begin{problem}[stronger variant of \cref{prob:simw-biclique-2}]\label{prob:simw-corona}
Does there exist a function\/ $\mu\colon\bbN\times\bbN\to\bbN$, such that every graph\/ $G$ of sim-width at most\/ $s$ that does not contain a\/ $t$-obstruction as an induced subgraph satisfies\/ $\yolov(G)\leq\mu(s,t)$?
\end{problem}

A positive answer to \cref{prob:simw-corona} would clearly imply a positive answer to \cref{prob:simw-biclique-2}, as any $t$-obstruction contains $K_{t,t}$ as an induced subgraph.

\Cref{prob:simw-biclique,prob:simw-corona} can be seen as a natural strengthening of our \cref{thm:biclique}.
We next discuss two possible natural generalizations of our \cref{thm:chibdd}.

We say that a class of graphs is \emph{polynomially\/ $\chi$-bounded} if it is $\chi$-bounded by a polynomial function.
While Briański, Davies, and Walczak~\cite{BDW24} recently showed that there exist $\chi$-bounded hereditary graph classes that are not polynomially $\chi$-bounded, the argument based on Ramsey's theorem showing that classes of bounded tree-independence number are $\chi$-bounded actually shows polynomial $\chi$-boundedness (see~\cite{DMS24a}).
On the other hand, our proof of \cref{thm:chibdd} results in non-polynomial $\chi$-bounding functions, leaving open the following.

\begin{problem}
Are graph classes of bounded induced matching treewidth polynomially\/ $\chi$-bounded?
\end{problem}

Another strengthening of \cref{thm:chibdd} would be given by a positive answer to the following question.

\begin{problem}
Are graph classes of bounded sim-width\/ $\chi$-bounded?
\end{problem}

We conclude with pointing out a certain drawback of sim-width: it is not clear how useful it is to obtain polynomial-time algorithms.
Some results in this direction were obtained by Bergougnoux, Korhonen, and Razgon~\cite{BKR23}.

\section*{Acknowledgements}

The main part of this work was done at the Structural Graph Theory workshop STWOR in September 2023.
This workshop was a part of STRUG: Structural Graph Theory Bootcamp, funded by the ``Excellence initiative -- research university (2020--2026)'' of University of Warsaw.
We thank the participants for a productive and inspiring atmosphere.

\bibliographystyle{plain}
\begin{sloppypar}
\bibliography{biblio}
\end{sloppypar}

\end{document}